\documentclass[11pt,leqno]{article}

\usepackage[latin1]{inputenc}
\usepackage{mathrsfs}
\usepackage{amsmath}
\usepackage{amssymb}
\usepackage{amsthm}
\usepackage{accents}

\usepackage[dvipdfmx]{graphicx}
\usepackage{graphicx}       
 
 \usepackage{tikz}

\usetikzlibrary{patterns}
\usetikzlibrary{intersections}

\numberwithin{equation}{section}

\newtheorem{theo}{Theorem}[section]
\newtheorem{prop}[theo]{Proposition}
\newtheorem{cor}[theo]{Corollary}
\newtheorem{lem}[theo]{Lemma}
\newtheorem{defi}[theo]{Definition}
\newtheorem{ass}[theo]{Assumption}

\newtheorem{rem}[theo]{Remark}

\title{Transversally strictly hyperbolic systems}

\author{Tatsuo Nishitani\footnote{Department of Mathematics,
Graduate School of Science, 
Osaka University, Osaka:  
nishitani@math.sci.osaka-u.ac.jp
}}

\date{}
\begin{document}

\maketitle

\def\RR{{\mathbb R}}
\def\CC{{\mathbb C}}
\def\lr#1{\langle{#1}\rangle}
\def\al{\alpha}
\def\be{\beta}
\def\ga{\gamma}
\def\si{\sigma}
\def\la{\lambda}
\def\bx{{\bar x}}
\def\bxi{{\bar \xi}}
\def\La{\Lambda}
\def\cL{{\mathcal L}}
\def\cE{{\mathcal E}}
\def\D{\partial}
\def\N{{\mathbb N}}
\def\xim{\lr{\xi}_{\mu}}
\newcommand{\mez}{\frac{1}{2}}
\newcommand{\Id}{\mathrm{Id}}
\newcommand{\re}{{\mathsf{Re\,}}}
\newcommand{\im}{{\mathsf{Im\,}}}
\newcommand{\eps}{\varepsilon }

\begin{abstract}
We  consider the Cauchy problem for first-order systems. Assuming that the set $\Sigma$ of singular points of the  characteristic variety is a smooth manifold and the characteristic values are real and semisimple we introduce a new class which is strictly hyperbolic in the directions transverse to $\Sigma$.  If the propagation cone and $\Sigma$ are compatible we prove, under some additional conditions,  that transversally strictly hyperbolic systems are strongly hyperbolic. On the other hand if the propagation cone is incompatible with $\Sigma$ then transversally strictly hyperbolic systems are much more involved which is discussed taking an interesting example.
\end{abstract}

 

\section{Introduction}

In this note, we  continue to study a new class of first-order systems  
 \begin{equation}
 \label{eq:siki}
 L (t, x,  D_t, D_x)  = D_{t } - \sum_{j=1}^d A_j(t, x) D_{x_j} = D_t - A(t, x, D_x) 
 \end{equation}
 which we call \emph{transversally strictly hyperbolic systems} in \cite{MeNi:KJM} and we 
 discuss whether transversally strictly hyperbolic systems
 are \emph{strongly hyperbolic}, which means by definition that  for all lower-order terms $B$ 
the Cauchy problem for $L + B$   is  well-posed in $C^{\infty}$. Here we use the notation  $D =-i\,\partial $ for partial derivatives. We make assumptions precise under which we work.
 \begin{ass}
\label{asshyp}\rm 
The coefficient matrices   $A_j(t, x)$ are $C^\infty$ in a neighborhood  $\Omega$ of the origin $(0,0)\in \RR^{1+d}$ and they act in $\CC^m$. 
Moreover, for all $(t, x, \xi)$ the eigenvalues of $A(t, x, \xi) = \sum_{j=1}^d \xi_j A_j(t, x)$ are real and semisimple. 
\end{ass}
We denote by $\Sigma$ the set of singular points of the characteristic variety 
of $L$, that is the set of  $(t, x, \tau, \xi) \in  T^*\Omega\backslash\{0\}\simeq \Omega\times ( \RR^{1+d} \backslash\{0\})$  such that 
$\det L(t, x, \tau, \xi) $ vanishes together with the first-order derivatives. 
\begin{ass}
\label{asshyp:2}\rm
We assume that $\Sigma$ is a smooth $C^\infty$-manifold in $T^*\Omega\backslash\{0\}$ and that, on 
each component of $\Sigma$ the dimension of 
${\rm Ker}\, L(t,x,\tau,\xi)$ is constant. 
\end{ass}
Note  that Assumption~\ref{asshyp} implies that at characteristic points, 
\begin{equation}
\label{eq:hogo}
{\rm dim}\,{\rm Ker}\,L(t, x, \tau, \xi)=\mbox{multiplicity of the eigenvalue $\tau $}. 
\end{equation}
\begin{defi}
\label{de:localization}\rm
Recall first the invariant definition of the 
\emph{localized system} (or \emph{localization}) $L_\rho$ at $\rho \in \Sigma$: 
\begin{equation}
L_\rho (\dot \rho)  = \varpi_\rho  ( L'(\rho) \cdot \dot \rho) \, \imath_\rho 
\end{equation} 
where  $ \imath_\rho  $ is the injection of ${\rm Ker}\, L(\rho)$ into $\CC^m$, $\varpi_\rho $ is the projection from
$\CC^m$ onto $\CC^m/{ \rm Range}\, L(\rho)$ and $L' (\rho):T_{\rho}(T^*\Omega)\to {\rm Hom}(\CC^m,\CC^m) $ is the derivative of $L$ at $\rho$. 
Since ${\rm Ker}\, L(\rho) \cap {\rm Range}\, L(\rho )= \{ 0 \}$  by Assumption~\ref{asshyp}, 
$L_\rho$ is a well-defined linear map on  ${\rm Ker}\,L(\rho)$ (see for example \cite[Chapter 4]{Ni:book}). Moreover $  \varpi_\rho  ( L'(\rho) \cdot T_{\rho}\Sigma) \, \imath_\rho=0$ by Assumptions  \ref{asshyp} and \ref{asshyp:2}.
\end{defi}
Note that $L_\rho$ is hyperbolic in the time direction, that is ${\rm det}\,L_{\rho}$ is a hyperbolic polynomial in the time direction (for more about the localization, see \cite[Section 4.1]{Ni:book}). We now introduce transversally strictly hyperbolic systems:
\begin{defi}\rm
If for all $\rho \in \Sigma$, $L_\rho (\dot \rho) $ is strictly hyperbolic in the time direction, 
 on   $~T^*\Omega/T_{\rho}\Sigma$,  then we call $L$ a transversally strictly hyperbolic system.
\end{defi}
Strictly hyperbolic system is smoothly symmetrizable (see \cite{Mi:book}) and repeating the construction of symmetrizer (\cite[Chapter  6]{Mi:book}) for strictly hyperbolic systems it is easily seen that transversally strictly hyperbolic systems are also symmetrizable in the following sense (\cite[Proposition 2.2]{MeNi:KJM}).
\begin{lem}
\label{lem:iosen}
 If $L$ is transversally strictly hyperbolic then $L$ is uniformly symmetrizable  (equivalently uniformly diagonalizable), that is there is a family of  symmetric positive definite matrices $S(t, x,  \xi)$, such that $S$ and $S^{-1}$ are uniformly bounded and $ S  A  $ is  symmetric. 
\end{lem}

The Cauchy problem for smoothly symmetrizable hyperbolic systems is $L^2$ well-posed (see \cite{FL}).  Smooth symmetrizability can be weakened as follows.
\begin{theo}[\cite{Me:X}]If there is a symmetrizer which is Lipschitz continuous in $(t,x,\xi)\in \RR^{1+d}\times S^{d-1}$ then the Cauchy problem for $L$ is $L^2$ well-posed and hence $L$ is a strongly hyperbolic system.
\end{theo}
However the symmetrizer for transversally strictly hyperbolic systems can not be chosen to be  Lipschitz continuous, not even continuous in general (see \cite[Lemma 3.3]{Me:X}).

Recall that even  at singular points of the characteristic vareity one can define a cone in the phase space, called the propagation cone,  which contains all possible directions along which the singularities could propagate.  In this note, in Section \ref{sec:compatible}, we study the Cauchy problem for transversally strictly hyperbolic systems with the propagation cone which is  compatible $\Sigma$ (the definition will be given below) and we prove, under some additional conditions, that the Cauchy problem is $C^{\infty}$ well-posed for any lower order term (Theorems \ref{thm:involutive_bis} and \ref{thm:symplectic}). If $\Sigma$ is an involutive or a symplectic manifold this is indeed the case. In Section \ref{sec:noncompatible} we study the case when the propagation cone is incompatible with  $\Sigma$. Taking an example containing a parameter proposed in \cite{Me:X} we show that the situation will be very complicated exhibiting well/ill posed  results (depending on the parameter) of the Cauchy problem (Theorems \ref{thm:illposed}, \ref{thm:lessone} and \ref{thm:eqone}).

Before closing this section we make some comments on the uniform symmetrizability. If the Cauchy problem is $L^2$ well-posed then the system is uniformly symmetrizable (see \cite[Theorem 1.2]{IvPe}) while the existence of a bounded symmetrizer does not imply in general the well-posedness of the Cauchy problem, even in $C^{\infty}$ (see  \cite{St}). On the other hand 
\begin{theo}[\cite{Ka}]
\label{thm:CaGe} Assume that $L(t,x,\tau,\xi)$ is uniformly symmetrizable then the Cauchy problem for $L+B$ is well-posed in the Gevrey class of order $1<s< 2$ for any $B$.
\end{theo}
Here we say that $f(x)\in \gamma^{(s)}(\RR^{d})$, the Gevrey class of order $s$, if for any compact set $K\subset \RR^{d}$ there exist $C>0, A>0$ such that we have
\[
|D^{\al}f(x)|\leq CA^{|\al|}|\al|!^s,\;\;x\in K,\;\;\forall \alpha\in\N^{d}.
\]
We denote $\gamma_0^{(s)}(\RR^d)=\gamma^{(s)}(\RR^d)\cap C_0^{\infty}(\RR^d)$.
A simple proof of Theorem \ref{thm:CaGe} by symmetrization is found in \cite{CNR}. Thanks to Lemma \ref{lem:iosen} we have
\begin{prop}
\label{pro:transtrictG}Assume that $L$ is transversally strictly hyperbolic. Then the Cauchy problem for $L+B$ is well-posed in the Gevrey class of order $1<s< 2$ for any $B$.
\end{prop}
We will see in Section \ref{sec:noncompatible} that the Gevrey order $2$ in Proposition \ref{pro:transtrictG} can not be improved in general (Theorem \ref{thm:illposed}).

\section{Propagation cone is compatible with $\Sigma$}
\label{sec:compatible}
 

Consider $L_{\rho}$ on $T^*\Omega$ not on $T^*\Omega/T_{\rho}\Sigma$, that is
\[
L_\rho (v)  = \varpi_\rho  ( L'(\rho) \cdot v) \, \imath_\rho,\quad v\in T_{\rho} (T^*\Omega)\simeq \RR^{1+d}\times \RR^{1+d}
\] 
 and hence $L_{\rho}$ is independent of directions parallel to $T_{\rho}\Sigma$.  With $h(t,x,\tau,\xi)={\rm det}\,L(t,x,\tau,\xi)$ in virtue of \eqref{eq:hogo} we have 
\[
{\rm det}\,L_{\rho}(t,x,\tau,\xi)=h_{\rho}(t,x,\tau,\xi)
\]
(see \cite[Lemma 4.2]{Ni:book}) where $h_{\rho}$ is the first non-vanishing term of the Taylor expansion of $h(t,x,\tau,\xi)$ at $\rho$ which is a homogeneous polynomial in $(t,x,\tau,\xi)$. Note that $h_{\rho}$ is a hyperbolic polynomial in the direction $\tau$ (\cite{IvPe}, \cite[Lemma 1.3.1]{Ho}). Recall that the hyperbolicity cone $\Gamma(h_{\rho})$ of $h_{\rho}$ is defined as the connected component of $-H_t$ where $H_t$ is the Hamilton vector field of $t$,  in
\[
\{(t,x,\tau,\xi)\in T_{\rho}(T^*\Omega)\mid h_{\rho}(t,x,\tau,\xi)\neq 0\}.
\]
The propagation cone $C(h_{\rho})$ of $h_{\rho}$ is defined by
\[
C(h_{\rho})=\{X=(t,x,\tau,\xi)\mid \sigma(X,Y)\leq 0, ~\forall Y\in \Gamma(h_{\rho})\}
\]
where $\sigma=d\tau\wedge dt+d\xi\wedge dx=d\tau\wedge dt+\sum_{j=1}^dd\xi_j\wedge dx_j$ is the symplectic two form on $T^*\Omega$. The propagation cone is the {\it minimal} cone including every 
bicharacteristic of $h$ which has $\rho$ as a limit point in the following sense (see \cite[Lemma 1.1.1]{KN}):
\begin{lem}Let $\rho\in T^*\Omega$ be a multiple characteristic of $h$. Assume that there are simple characteristics $\rho_j$ of $h$ and positive numbers $\gamma_j$ such that
\[
\rho_j\to \rho \quad\mbox{and}\quad \gamma_jH_h(\rho_j)\to X(\neq 0)\quad j\to\infty.
\]
Then $X\in C(h_{\rho})$. Here $H_h$ denotes the Hamilton vector field of $h$.
\end{lem}
Since $H_h(\rho)$ is the  direction along which the singularities propagate (if $\rho$ is a simple characteristic) the mutual relative position of $C(h_{\rho})$ and  $T_{\rho}\Sigma$ would be significant for the Cauchy problem and taking this into account  we introduce the following definition.
\begin{defi}
\label{de:seigou}\rm The propagation cone is said to be compatible with $\Sigma$ if, on each component $\Sigma'$ of $\Sigma$, either $C(h_{\rho})\subset T_{\rho}\Sigma$, $\rho\in \Sigma'$ or $C(h_{\rho})\cap T_{\rho}\Sigma=\{0\}$, $\rho\in \Sigma'$ holds, otherwise we say  that the propagation cone is incompatible with $\Sigma$.
\end{defi}
\begin{defi}
\label{def:tekisetu}\rm We say that the Cauchy problem for $L$ is $C^{\infty}$ well-posed near the origin if there exist a positive $\epsilon>0$ and a neighborhood $\omega$ of the origin such that for any $|\bar t|\leq \epsilon$ and $f(x)\in C_0^{\infty}(\omega)$ vanishing in $t<\bar t$ there is a unique $u\in C^{\infty}(\omega)$ vanishing in $t<\bar t$ which satisfies $Lu=f$ in $\omega$.
\end{defi}
 In this section we discuss the case that the propagation cone  is compatible with  $\Sigma$.
\begin{theo}[\cite{MeNi:KJM}]
\label{thm:involutive_bis}Assume that $L$ is transversally strictly hyperbolic and $C(h_{\rho})\subset T_{\rho}\Sigma$ at every $\rho\in \Sigma$. Then the Cauchy problem for $L+B $  is $C^{\infty}$ well-posed for any $B(t,x)$ near the origin.
\end{theo}
This follows from \cite[Theorem 1.6]{ MeNi:KJM} thanks to the following
\begin{lem}[\cite{Ni:JdAM}]
\label{lem:zao}$C(h_{\rho})$ is contained in $T_{\rho}\Sigma$ at every  $\rho\in \Sigma$ if and only if $\Sigma$ is an involutive manifold.
\end{lem}
We turn to study another compatible case, that is  $C(h_{\rho})\cap T_{\rho}\Sigma=\{0\}$. We first recall 
\begin{lem}{\rm(\cite[Lemma 1.1.3]{KN})}
\label{lem:KNsp}
The following conditions are equivalent.
\begin{description}
\item{\rm(i)} $C(h_{\rho})\cap T_{\rho}\Sigma=\{0\}$,
\item{\rm(ii)} there is a hyperplane $H\subset T_{\rho}(T^*\Omega)$ such that 
\[
H\cap C(h_{\rho})=\{0\},\quad H\supset T_{\rho}\Sigma+\lr{H_t},
\]
\item{\rm(iii)} $\Gamma(h_{\rho})\cap (T_{\rho}\Sigma)^{\sigma}\cap \lr{H_t}^{\sigma}\neq\emptyset$,
\item{\rm(iv)} $\Gamma(h_{\rho})\cap (T_{\rho}\Sigma)^{\sigma}\neq\emptyset$
\end{description}
where $\lr{H_t}=\{\lambda H_t\mid \lambda\in\RR\}$ and $S^{\sigma}=\{X\in T_{\rho}(T^*\Omega)\mid \sigma(X,Y)=0, \forall Y\in S\}$.
\end{lem}

Denote by $\pi$ the projection such that $\pi(t,x,\tau,\xi)=(t,x)$.  
\begin{theo}
\label{thm:symplectic}Assume that $L$ is transversally strictly hyperbolic and that at every $\rho\in\Sigma\cap \pi^{-1}(0,0)$ there is a hyperplane $H\subset T_{\rho}(T^*\Omega)$ such that
\begin{equation}
\label{eq:odan}
C(h_{\rho})\cap H=\{0\},\quad H\supset T_{\rho}\Sigma+\lr{H_t}\supset H^{\sigma}.
\end{equation}
 Then the Cauchy problem for $L+B $   is $C^{\infty}$ well-posed  for any $B(t,x)$ near the origin. 
\end{theo}
Here we give a sufficient condition for \eqref{eq:odan}.
\begin{lem}
\label{lem:dtzero}If $dt\big( T_{\rho}\Sigma\cap (T_{\rho}\Sigma)^{\sigma}\big)=0$ then there exists a hyperplane $H\subset T_{\rho}(T^*\Omega)$ such that \eqref{eq:odan} is verified.
\end{lem}
The proof of this lemma will be given later. If $\Sigma$ is a symplectic manifold, that is $T_{\rho}\Sigma\cap (T_{\rho}\Sigma)^{\sigma}=\{0\}$, $\rho\in \Sigma$  then $dt\big( T_{\rho}\Sigma\cap (T_{\rho}\Sigma)^{\sigma}\big)=0$ is obvious and hence  
\begin{cor}
\label{cor:symcase}Assume that $L$ is transversally strictly hyperbolic and $\Sigma$ is either involutive  or symplectic manifold. Then the Cauchy problem for $L+B $  is $C^{\infty}$ well-posed for any $B(t,x)$ near the origin.
\end{cor}
\begin{theo}
\label{thm:compo}Let $L$ be transversally strictly hyperbolic. Assume that at every $\rho\in\Sigma\cap \pi^{-1}(0,0)$ either there is a hyperplane $H\subset T_{\rho}(T^*\Omega)$ satisfying 
\eqref{eq:odan} or there is a neighborhood $U$ of $\rho$ such that $C(h_{\rho})\subset T_{\rho}\Sigma$, $\rho\in U\cap\Sigma$.  Then the Cauchy problem for $L+B $   is $C^{\infty}$ well-posed  for any $B(t,x)$ near the origin. 
\end{theo}
%


\noindent
{ Proof of Theorem \ref{thm:symplectic}: Let $B(t,x)$ be any $m\times m$ matrix valued $C^{\infty}$ function near $(0,0)$. 
In order to prove that the Cauchy problem for $L+B$ is locally solvable near $(0,0)$ in $C^{\infty}$ it is enough to show the existence of a parametrix with finite propagation speed of wave front sets (see \cite[Appendix]{Ni:KJM}, \cite[Appendix]{KN})  at every $(0,0,\xi)$ with $|\xi|=1$ (from now on, we say ''parametrix with fps'' for short).  Let   $|{\bar\xi}|=1$ be arbitrarily fixed 
and denote   by ${\bar \tau}_k$, $k=1,\ldots,l$ the distinct eigenvalues  of $A(0, 0, {\bar \xi})$. Then one can find  $P(t,x,\xi)\in S^0_{phg}$,   $P(t, x, \xi)\sim \sum_{j=0}^{\infty}P_j(t,x,\xi)$ defined for small $|t|$ with ${\rm det}\,P_0(0,0,{\bar\xi})\neq 0$ (for the definition of the symbol class $S_{phg}$, see \cite[Chapter XVIII]{Ho:book:3}) such that  
\[
 (L+B)P\equiv P\,{\rm diag} (D_t-{ A}_1,\ldots, D_t-{ A}_l) \;\quad \text{near}\;\;(0,0,{\bar \xi})
\]
that is,  $C(t,x,\xi)$, the symbol of $(L+B)P-P\,{\rm diag} (D_t-{ A}_1,\ldots, D_t-{ A}_l)$ is in $S^{-\infty}$ in a conic neighborhood of $(0,0,{\bar\xi})$ uniformly in $t$ for small  $|t|$ where ${ A}_k(t,x,\xi)\in S^1_{phg}$ are $m_k\times m_k$ matrix valued symbols, ${ A}_k\sim \sum_{j=0}^{\infty}{ A}_{kj}$ and the eigenvalues of ${ A}_{k0}(t,x,\xi)$ are close to ${\bar \tau}_k$ near $(0,0,{\bar\xi})$, where $m_k$ is the 
multiplicity of ${\bar\tau}_k$ (for a proof, see  \cite{Wa}). From \cite[Proposition A.4]{KN} it follows that $L+B$ has a parametrix with fps at $(0,0,{\bar\xi})$ if ${\rm diag} (D_t-{ A}_1,\ldots, D_t-{ A}_l)$ does. It is clear that if each $D_t-{ A}_j$ has a parametrix with fps $G_j$ at $(0,0,{\bar \xi})$ then $G={\rm diag}(G_1,\ldots,G_l)$ is a parametrix with fps of ${\rm diag} (D_t-{ A}_1,\ldots, D_t-{ A}_l)$ at $(0,0,{\bar \xi})$. Therefore it suffices to show the existence of a parametrix with fps of each $D_t-{ A}_j$ at $(0,0,{\bar\xi})$.  Since 
\begin{align*}
&L(t,x,\tau, \xi)P_0(t,x,\xi)
\\=P_0(t,x,\xi)&\,{\rm diag}(\tau-{ A}_{10}(t,x,\xi),\ldots, \tau-{ A}_{l0}(t,x,\xi))
\end{align*}
 in a conic neighborhood of $(0,0, {\bar\xi})$, denoting by $\Sigma_{(j)}$, the component of $\rho_j=(0,0,{\bar \tau}_j,{\bar\xi})$ in $\Sigma$, the Assumption \ref{asshyp:2} implies that ${\rm dim}\,{\rm Ker}\,(\tau-{ A}_{j0}(t,x,\xi))$ is constant on $\Sigma_{(j)}$ which is $m_j$ by \eqref{eq:hogo}.  Then one has $\tau-{ A}_{j0}(t,x,\xi)=O$ on $\Sigma_{(j)}$  and hence  ${ A}_{j0}(t,x,\xi)-\mu_j(t,x,\xi)I_{m_j}=O$ on $\Sigma'_{(j)}$ which is the projection of $\Sigma_{(j)}$ off $\tau$ coordinate where $\mu_j(t,x,\xi)={\rm tr}A_{j0}(t,x,\xi)/m_j$.  Note that $\Sigma_{(j)}$ is given by $\tau=\mu_j(t,x,\xi)$, $b_{ji}(t,x,\xi)=0$, $i=1,\ldots,k_j$ near $\rho_j$ where $db_{ji}$ are linearly 
independent at $(0,0, {\bar\xi})$. From the assumption there exist hyperplanes $H_j\subset T_{\rho_j}(T^*\Omega)$ such that
\begin{equation}
\label{eq:kawa}
C(h_{\rho_j})\cap H_j=\{0\},\quad H_j\supset T_{\rho_j}\Sigma_{(j)}+\lr{H_t}\supset H_j^{\sigma}.
\end{equation}
Denote $L_j=\tau-{ A}_{j0}$ so that $P_0^{-1}(\tau-A)P_0={\rm diag}(L_1,\ldots,L_l)$. Then from definition  there is a non-singular $m_j\times m_j$ matrix $T$ such that $
T^{-1}L_{\rho_j}T=(L_{j})_{\rho_j}$ from which it follows that
\begin{equation}
\label{eq:yama}
\text{$(L_j)_{\rho_j}$ is strictly hyperbolic  on $T^{*}\Omega/T_{\rho_j}\Sigma_{(j)}$}.
\end{equation}
We show that $D_t-{ A}_j$ has a parametrix with fps at $(0,0,{\bar\xi})$. Denote ${ A}_j$ by ${ A}$, $\Sigma_{(j)}$ by $\Sigma$, $\Sigma_{(j)}'$ by $\Sigma'$, $b_{ji}(t,x,\xi)$ by $b_i(t,x,\xi)$, $m_j$ by $m$ and $\mu_j$ by $\mu$ again,  dropping $j$. Write
\[
D_t-{ A}(t,x,D_x)=\big(D_t-\mu(t,x,D_x)\big)I-{\hat A}(t,x,D_x)
\]
where ${\hat A}(t,x,\xi)\sim {\hat A}_1(t,x,\xi)+{\hat A}_0(t,x,\xi)+\cdots$ and ${\hat A}_1(t,x,\xi)$, homogeneous of degree $1$ in $\xi$, satisfies ${\hat A}_1(t,x,\xi)=O$  on $\Sigma'$ near  $(0,0,{\bar\xi})$. Note that $\mu(t,x,\xi)$ is real valued for the eigenvalues are real. Let $S(t',t)$ be  the solution operator of the Cauchy problem 
\[
D_tu-\mu(t,x,D_x)u=0,\quad u(t', x)=\phi(x)
\]
such that $\phi\mapsto u(t)$. Then it is clear that $S(t,0)(D_t-\mu)S(0,t)=D_t$
and there exists ${\tilde A}\in S^1_{phg}$ such that $
S(t,0){\hat A}S(0,t)={\tilde A}+R$ where $R$ maps $H^{-\infty}=\cup_s H^s(\RR^d)$ into $H^{\infty}=\cap_sH^s(\RR^d)$ (see for example \cite[Chapter VIII]{Tay} ). Therefore one obtains
\begin{equation}
\label{eq:appF}
\big(D_t-\mu(t,x,D)I-{\hat A}\big)S(0,t)=S(0,t)\big(D_t-{\tilde A}\big)+S(0,t)R.
\end{equation}
Note that $S(0,t)$ is a Fourier integral operator associated with the canonical transformation $\chi_t(x,\xi)=\exp{(-tH_{\mu})}(x,\xi)$ and hence ${\tilde A}_1(t,x,\xi)={\hat A}_1(t,\chi_t(x,\xi))$. Therefore , ${\tilde \Sigma}$ the image of $\Sigma$, is given by the equations
\[
\tau=0,\quad {\tilde b}_1(t,x,\xi)=\cdots={\tilde b}_k(t,x,\xi)=0,\;\;{\tilde b}_i(t,x,\xi)=b_i(t,\chi_t(x,\xi))
\]
and ${\tilde A}_1=O$ on ${\tilde \Sigma}'=\{{\tilde b}_1=\cdots={\tilde b}_k=0\}$ where ${\tilde b}_i(t,x,\xi)$ are homogeneous of degree $1$ in $\xi$. From \eqref{eq:appF} and \cite[Proposition A.5]{Ni:book} it follows that $(D_t-\mu) I-{\hat A}$ has a parametrix with fps at $(0,0,{\bar\xi})$ if $D_t-{\tilde A}$ does. We next show that $D_t-{\tilde A}$ has  a parametrix with fps at $(0,0,{\bar\xi})$. To simplify notation write $L=D_t-{\tilde A}_1$ and ${\tilde B}(t,x,\xi)={\tilde A}_1(t,x,\xi)-{\tilde A}(t,x,\xi)$ such that $D_t-{\tilde A}=L+{\tilde B}$ where ${\tilde B}(t,x,\xi)\in S^0_{phg}$ and ${\tilde B}\sim {\tilde B}_0+ {\tilde B}_{-1}+\cdots$.

Denote by $M(t,x,\tau,\xi)$ the cofactor matrix of $L(t,x,\tau,\xi)$. Then by Proposition \cite[Proposition A.2]{Ni:book} we see that $L+{\tilde B}$ has a parametrix with fps at $(0,0,{\bar\xi})$ if  ${\mathcal P}(t,x,D_t,D_x)=\big(L(t,x,D_t,D_x)+{\tilde B}(t,x, D_x)\big)M(t,x,D_t,D_x)$ does. 
\begin{lem}
\label{lem:kijyu:b}Let $\rho=(0,0, 0,{\bar\xi})$. One can write
\begin{equation}
\label{eq:kijyu:b}
{\mathcal P}(t,x,\tau,\xi)=h(t,x,\tau,\xi)I+P_{m-1}+\cdots+P_0
\end{equation}
where $h(t,x,\tau,\xi)={\rm det}\,L(t,x,\tau,\xi)$ and
\begin{itemize}
\item[\rm(i)] there is a hyperplane $H\subset T_{\rho}^*(\Omega)$ verifying  $C_{\rho}(h_{\rho})\cap H=\{0\}$ and $H\supset T_{\rho}{\tilde \Sigma}+\lr{H_t}\supset H^{\sigma}$, 
\item[\rm(ii)] $h(t,x,\tau,\xi)$ is a polynomial in $\tau$ and homogeneous of degree $m$ in $(\tau,\xi)$ which is hyperbolic  in $t$ direction near $(0,0,{\bar\xi})$ and the localization $h_{\rho}(t,x,\tau,\xi)$ is strictly hyperbolic on $T^*_{\rho}\Omega/T_{\rho}{\tilde \Sigma}$, 
\item[\rm(iii)]
$P_j(t,x,\tau,\xi)=\sum_{k=1}^{m}P_{jk}(t,x,\xi)\tau^{m-k}$ where $
P_{jk}(t,x,\xi)$  are positively homogeneous of degree $j-m+k$ in $\xi$ if $j\geq 1$, and 
$P_{0k}(t,x,\xi)\in S^{-m+k}_{phg}$, and $P_j(t,x,\tau,\xi)$ $(j\geq 1)$ vanishes of order $m-2j$ on ${\tilde \Sigma}$ near $\rho$.
\end{itemize}
\end{lem}
\begin{proof} The assertions (i) and (ii) follows from \eqref{eq:kawa} and \eqref{eq:yama} immediately because the canonical transformation $\chi_t(x,\xi)$ leaves the coordinate $t$ invariant. It remains to show (iii). Since ${\tilde A}_1(t,x,\xi)=O$ on ${\tilde \Sigma}$ one can write near $(0,0,{\bar\xi})$
\[
L(t,x,\tau,\xi)=\tau I+\sum_{j=1}^k{\tilde A}_{1j}(t,x,\xi)b_j(t,x,\xi)
\]
and hence $M(t,x,\tau,\xi)$ is a homogeneous polynomial in $(\tau,b_1,\ldots,b_k)$ of degree $m-1$ with coefficients which are homogeneous of degree $0$ in $\xi$. Since 
\begin{align*}
P_{m-j}(t,x,\tau,\xi)=\sum_{l+|\alpha|=j}\frac{(-i)^{j}}{l!\alpha!}\D_{\tau}^l\D_{\xi}^{\alpha}L(t,x,\tau,\xi)\D_{t}^l\D_{x}^{\alpha}M(t,x,\tau,\xi)\\
+\sum_{l+|\alpha|+k+1=j}\frac{(-i)^{j-1}}{l!\alpha!}\D_{\tau}^l\D_{\xi}^{\alpha}{\tilde B}_k(t,x,\xi)\D_{t}^l\D_{x}^{\alpha}M(t,x,\tau,\xi)
\end{align*}
for $1\leq j\leq m-1$ and $\D_{t}^l\D_{x}^{\alpha}M(t,x,\tau,\xi)$ vanishes on ${\tilde \Sigma}$ of order $m-1-l-|\alpha|$ so that $P_{m-j}$ vanishes of order $m-1-j$ there. Then the assertion is clear because $m-2j\leq m-1-j$ for $j\geq 1$.
\end{proof}
We can now prove the existence of a parametrix with fps of ${\mathcal P}$ at $({\bar t},0,{\bar\xi})$ with $|{\bar t}|\leq \epsilon$ for a small $\epsilon>0$ applying \cite[Proposition 6.1, Proposition 6.2]{Ni:KJM} with an obvious modification.   
\begin{lem}
\label{lem:bosai}Assume that {\rm (i)} in Lemma \ref{lem:kijyu:b} holds. Then one can find a neighborhood $U$ of $\rho$ such that ${\tilde \Sigma}\cap U$ is given by $\tau=0$, $b_j(t,x,\xi)=0$, $j=1,\ldots,k$ where $db_j$ are linearly independent at $\rho$ such that
\[
\{\tau,b_1\}\neq 0,\quad \{\tau,b_j\}=\{b_1,b_j\}=0,\;\;j=2,\ldots,k,\;\;\text{at $\rho$}.
\]
\end{lem}
\begin{proof}Let ${\tilde \Sigma}$ be given by $\tau=0$, ${\tilde b}_j=0$ near $\rho$. Note that $\{\tau,{\tilde b}_j\}\neq 0$ at $\rho$ for some $j$. Otherwise one has $H_{\tau}\in T_{\rho}{\tilde \Sigma}\cap C(h_{\rho})$ which is, by Lemma  \ref{lem:KNsp}, a contradiction. Thus we may assume $\{\tau,{\tilde b}_1\}\neq 0$ at $\rho$. Replacing ${\tilde b}_j$ by ${\tilde b}_j-c_jb_1$, $j=2,\ldots,k$ with suitable $c_j\in \RR$ one can assume $\{\tau,{\tilde b}_j\}=0$ at $\rho$ for $j=2,\ldots,k$. Since $H\supset T_{\rho}{\tilde \Sigma}$ hence $H$ is given by $\{db_1(X)=0\}$ where $b_1=\al_0\tau+\sum_{j=1}^k\al_j{\tilde b}_j$  with some $\al_j\in\RR$. Note that $\al_0=0$ because $H\supset \lr{H_t}$. If $\al_1=0$ then $\{\tau, b_1\}=0$ at $\rho$ which implies that $H_{\tau}\in C(h_{\rho})\cap H$ contradicting the assumption.  Therefore ${\tilde \Sigma}$ is given by $\tau=0$, $b_1=0$, ${\tilde b}_j=0$, $j=2,\ldots,k$ near $\rho$. Note that $T_{\rho}{\tilde \Sigma}+\lr{H_t}\supset H^{\sigma}$ implies that $H\supset (T_{\rho}{\tilde \Sigma})^{\sigma}\cap \lr{H_t}^{\sigma}$. Since $H_{{\tilde b}_j}\in (T_{\rho}{\tilde \Sigma})^{\sigma}\cap \lr{H_t}^{\sigma}$, $j=2,\ldots,k$ it follows that $\{b_1,{\tilde b}_j\}=0$ at $\rho$ which proves the assertion putting $b_j={\tilde b}_j$, $j=2,\ldots,k$.
\end{proof}
In \cite{Ni:KJM} for a single operator of the form 
\[
P=h(t,x,D_t,D_x)+p_{m-1}(t,x,D_t,D_x)+p_{m-2}(t,x,D_t,D_x)+\cdots
\]
where $h(t,x,\tau,\xi)$ and $p_{m-j}(t,x,\tau,\xi)$ satisfy (i) (in reality, \cite[Lemma 2.1]{Ni:KJM} resulting from (i) thanks to Lemma \ref{lem:bosai} above), (ii) and (iii) in Lemma \ref{lem:kijyu:b} the existence of a parametrix with fps was proved by reducing the equation $Pu=f$ to a second order system with unknowns $(u^{\sigma}_j)$, $1\leq j\leq [m/2]$, $\sigma$; permutations  on $(1,2,\ldots,m)$ (see \cite[Proposition 7.2]{Ni:KJM}) for which one can apply \cite[ Corollary 6.1]{Ni:KJM}. In the present case, by exactly the same procedure one can reduce ${\mathcal P}U=F$, $U=(u_1,\ldots,u_m)$ to a second order system with unknowns $(u_{ij}^{\sigma})$, $1\leq i\leq m$, $1\leq j\leq [m/2]$, $\sigma$; permutations  on $(1,2,\ldots,m)$, here the size of the reduced system is $m$ times that of $P$,  for which one can apply \cite[Corollary 6.1]{Ni:KJM}. We summarize
\begin{prop}
\label{prop:symplectic}Assume that $L$ is transversally strictly hyperbolic and at every $\rho\in \Sigma\cap \pi^{-1}(0,0)$ there is a hyperplane $H\subset T_{\rho}(T^*\Omega)$ satisfying \eqref{eq:odan}. Then for any $B(t,x)$ there exist $\epsilon_0>0$ and a neighborhood $U$ of $(0,0)$ such that for any $|{\bar t}|\leq \epsilon_0$ and for  any $f\in C_0^{\infty}(U)$ vanishing in $t\leq {\bar t}$ there exists $u\in C^{\infty}(U)$ vanishing in $t\leq {\bar t}$ and satisfying $(L+B)u=f$ in $U$.
\end{prop}
We turn to the uniqueness of solution. Denote $
(L+B)^*={\hat L}(t,x,D_t,D_x)+{\hat B}(t,x)$ where ${\hat L}(t,x,\tau,\xi)=L(t,x,\tau,\xi)^*$ and ${\hat B}=B^*(t,x)-\sum_{j=1}^dD_{x_j}A_j^*(t,x)$.
 Since  ${\rm det}\,{\hat L}(t,x,\tau,\xi)=\overline{h(t,x,\tau,\xi)}=h(t,x,\tau,\xi)$ it is clear that ${\hat L}$ has the same $\Sigma$ and satisfies Assumptions \ref{asshyp} and \ref{asshyp:2} and transversally strictly hyperbolic. Choose a new system of local coordinates $({\tilde t},{\tilde x})$ around $(0,0)$ such that
\[
{\tilde t}=t+\epsilon\sum_{j=1}^dx_j^2,\;\;{\tilde x}_j=x_j,\;\;j=1,2,\ldots,d
\]
which is so called Holmgren transform where $\epsilon>0$ is a small positive constant. In these coordinates the symbol of ${\hat L}$ is given by
\begin{equation}
\label{eq:rekyu}
\begin{split}
{\hat L}({\tilde t}-\epsilon |{\tilde x}|^2,{\tilde x},{\tilde \tau}, {\tilde \xi}+2\epsilon {\tilde\tau}{\tilde x})
=C\big({\tilde \tau}-C^{-1}\sum_{j=1}^dA_j^*(t,x){\tilde\xi}_j\big)\\
=C\big({\tilde \tau}-C^{-1}A^*(t,x,{\tilde \xi})\big)
\end{split}
\end{equation}
where $C=I+2\epsilon\sum A^*_j (t,x){\tilde x}_j$ which is non-singular for small $\epsilon>0$ and small $|x|$. It is known that all eigenvalues of $C^{-1}A^*(t,x,{\tilde\xi})$ are real (see \cite{Wei}, \cite{Lax} also \cite[Lemma 6.5]{Mi:book}). Recall that $\Sigma_{(j)}$, the component of $\rho_j$ in $\Sigma$, is given by $\tau=\mu_j(t,x,\xi)$, $b_{ji}(t,x,\xi)=0$, $i=1,\ldots,k_j$ near $\rho_j$. Consider the equation ${\tilde \tau}=\mu_j({\tilde t}-\epsilon |{\tilde x}|^2,{\tilde x}, {\tilde \xi}+2\epsilon {\tilde\tau}{\tilde x})$ which can be  solved with respect to ${\tilde \tau}$ near $\rho_j$, and  we denote it by ${\tilde \tau}={\tilde \mu}_j({\tilde t},{\tilde x},{\tilde\xi})$. We now denote by ${\tilde \Sigma}_{(j)}$, the manifold given by
\[
{\tilde \tau}={\tilde \mu}_j({\tilde t},{\tilde x},{\tilde \xi}),\quad {\tilde b}_{ji}({\tilde t},{\tilde x},{\tilde \xi})=b_{ji}({\tilde t}-\epsilon |{\tilde x}|^2,{\tilde x}, {\tilde \xi}+2\epsilon {\tilde \mu}_j({\tilde t},{\tilde x},{\tilde \xi}){\tilde x})=0.
\]
Since ${\rm Ker}\,{\hat L}(t,x,\tau,\xi)={\rm Ker}\, \big({\tilde \tau}-C^{-1}\sum_{j=1}^dA^*_j(t,x){\tilde\xi}_j\big)$ by \eqref{eq:rekyu} it follows that 
\[
{\rm dim}\,{\rm Ker}\, \big({\tilde \tau}-C^{-1}\sum_{j=1}^dA^*_j(t,x){\tilde\xi}_j\big)=m_j\quad \mbox{on}\quad {\tilde \Sigma}_{(j)}.
\]
Therefore Assumptions \ref{asshyp} and \ref{asshyp:2} are satisfied for 
\[
{\tilde L}=D_{{\tilde t}}-C^{-1}\sum A^*_j(t,x)D_{{\tilde x}_j}
\]
and  ${\tilde L}$ is transversally strictly hyperbolic since Definition \ref{de:localization} is coordinate free. Moreover   at every $\rho\in{\tilde \Sigma}\cap \pi^{-1}(0,0)$ there is $H\subset T_{\rho}^*(\Omega)$ satisfying \eqref{eq:odan} because $H_{\tilde t}=H_t$ at  $(0,0)$.  
\begin{lem}
\label{lem:itii}Assume that the assumption in Theorem \ref{thm:symplectic} is satisfied. Then for any $B(t,x)$ one can find a neighborhood $U$ of $(0,0)$, positive numbers ${\bar \epsilon}>0, \epsilon>0$ such that for any $f(t,x)\in C_0^{\infty}(U)$ with ${\rm supp}\,f\subset \{(t,x)\mid t\leq {\bar\epsilon}-\epsilon |x|^2\}$ there exists $v\in C^1(U)$ with ${\rm supp}\,v\subset\{(t,x)\mid t\leq {\bar\epsilon}-\epsilon |x|^2\}$ which satisfies $(L+B)^*v=f$.
\end{lem}
\begin{proof}We apply Proposition \ref{prop:symplectic} to ${\tilde L}+C^{-1}{\hat B}$ and $C^{-1}f$ with the reversed time direction and ${\bar t}={\bar \epsilon}$. Turning back to the coordinates $(t,x)$ we conclude the assertion.
\end{proof}
Assume that $C^1$ function $u$ vanishing in $t<0$ satisfies $(L+B)u=0$ in a neighborhood of $(0,0)$. Take $v$ in Lemma \ref{lem:itii} and consider
\begin{align*}
0=\int_0^{\bar\epsilon}\int_{\RR^d}(L+B)u\cdot v dtdx=\int_0^{\bar\epsilon}\int_{\RR^d}u\cdot (L+B)^*v dtdx\\
=\int_0^{\bar\epsilon}\int_{\RR^d}u\cdot fdtdx.
\end{align*}
Since $f\in C_0^{\infty}$ with ${\rm supp}\,\subset\{(t,x)\mid t\leq {\bar\epsilon}-\epsilon|x|^2\}$ is arbitrary we conclude that $u=0$ in the set $\{(t,x)\mid 0\leq t\leq {\bar\epsilon}-\epsilon|x|^2\}$ which proves the uniqueness of solution.

\medskip

\noindent
{ Proof of Theorem \ref{thm:compo}: From \cite[Theorem 1.6]{MeNi:KJM} and Theorem \ref{thm:symplectic} (or rather from the proofs), at every $\rho\in\Sigma\cap \pi^{-1}(0,0)$ there exists a parametrix with fps.

\medskip
\begin{rem}
\label{rem:yoso}\rm It is expected that Theorem \ref{thm:symplectic} holds under the assumption $C(h_{\rho})\cap T_{\rho}\Sigma=\{0\}$, that is under the equivalent assumption of the existence of hyperplane $H\subset T_{\rho}\Sigma$ such that 
\begin{equation}
\label{eq:doticon}
C(h_{\rho})\cap H=\{0\}, \quad H\supset T_{\rho}\Sigma+\lr{H_t}.
\end{equation}
Assuming \eqref{eq:doticon} we can always choose a suitable local canonical coordinates system $(t,x,\tau,\xi)$ in $T^*\Omega$   with which the condition \eqref{eq:odan} holds (see \cite{Ni:OJM}). But in studying the well-posedness of the Cauchy problem, not all canonical transformations are allowed since one can only use the canonical transformation such that the associated Fourier integral operator preserves the causality, that is such $F$ satisfying  $Fu=0$ in $t<\bar t$ if $u=0$ in $t<\bar t$ (this is a special aspect of the Cauchy problem). In general, under the assumption \eqref{eq:doticon} one can not choose such causality preserving  canonical coordinates system with which \eqref{eq:odan} holds. This is the main reason why the study of the Cauchy problem under the condition \eqref{eq:doticon} is not so straightforward.  
\end{rem}
 \begin{rem}
 \label{rem:zohiko}\rm Contrary to the case studied in Theorem \ref{thm:involutive_bis}, in the case treated in Theorem \ref{thm:symplectic}, solutions to the Cauchy problem for $L+B$ lose much regularity depending on $B$.  We give the simplest example (see for example \cite[Section 1.3]{Ni:book})
 \[
 L=\frac{\D}{\D t}+\left(\begin{array}{cc}
0&1\\
t^2&0
\end{array}\right)\frac{\D}{\D x}+B(t,x)
\]
where $h_{\rho}=\tau^2-{\bar\xi}t^2$ and $\Sigma=\{\tau=0,t=0\}$ is a symplectic manifold.
 \end{rem}
\begin{rem}
\rm If  $L$ is   transversally strictly hyperbolic system, by definition $h(t,x,\tau,\xi)={\rm det}\,L(t,x,\tau,\xi)$ is strictly hyperbolic polynomial in $\tau$ on $~T^*\Omega/T_{\rho}\Sigma$, while  $h(t,x,D_t,D_x)$ is not strongly hyperbolic anymore if $m\geq 3$  by \cite[Theorem 1.4 ]{IvPe} whatever the mutual relative position of $C(h_{\rho})$ and  $T_{\rho}\Sigma$ is. 
\end{rem}

\noindent
Proof of Lemma \ref{lem:dtzero}: Without restrictions one can assume that $\Sigma$ is given by $\tau=0$, $b_j(t,x,\xi)=0$, $j=1,\ldots,k$ near $\rho$ where $db_j$ are linearly independent at $\rho$. Note that $\{\tau,b_j\}\neq 0$ at $\rho$ for some $j$. Otherwise one would have  $H_{\tau}\in T_{\rho}\Sigma\cap (T_{\rho}\Sigma)^{\sigma}$ since $(T_{\rho}\Sigma)^{\sigma}$ is spanned by $H_{\tau}$, $H_{b_j}$, $j=1,\ldots,k$ which is a contradiction since $dt(H_{\tau})\neq 0$. As in the proof of Lemma \ref{lem:bosai} we can assume $\{\tau,b_1\}\neq 0$, $\{\tau,b_j\}=0$ for $j=2,\ldots,k$ at $\rho$. If $\{b_1,b_j\}=0$ at $\rho$ for $j=2,\ldots,k$ then $H=\{db_1(X)=0\}$ is a desired hyperplane. Repeating the same argument one can find $\ell ~(\geq 2)$ such that 
\[
\{b_i,b_{i+1}\}\neq 0,\;\;i=0,\ldots,\ell-1,\quad \{b_i,b_j\}=0,\;\;j=i+2,\ldots,k
\]
at $\rho$ where $b_0=\tau$. We show that $\ell$ is odd. In fact if $\ell$ is even then the equations $\sum_{j=0}^{\ell}\{b_i,b_j\}(\rho)c_j=0$, $i=0,1,\ldots,\ell$ has a solution $(c_0,c_1,\ldots,c_{\ell})$ with $c_0\neq 0$. This implies that
\[
\sum_{j=0}^{\ell}c_jH_{b_j}\in T_{\rho}\Sigma\cap (T_{\rho}\Sigma)^{\sigma},\quad dt\big(\sum_{j=0}^{\ell}c_jH_{b_j}\big)\neq 0
\]
which is a contradiction. Since $\ell$ is odd the equations $\sum_{j=1}^{\ell}\{b_i,b_j\}(\rho)c_j=0$, $i=1,\ldots,\ell$ has a non-trivial solution $(c_1,\ldots,c_{\ell})$ where it is easy to see $c_1\neq 0$. Thus $\Sigma$ is given by $\tau=0$, ${\tilde b}_1=\sum_{j=1}^{\ell}c_jb_j=0$, $b_j=0$, $j=2,\ldots,k$. We show that  $H=\{d{\tilde b}_1(X)=0\}$ is a desired plane. Since either $H_{{\tilde b}_1}\in \Gamma(h_{\rho})$ or $-H_{{\tilde b}_1}\in \Gamma(h_{\rho})$ it follows that $C(h_{\rho})\cap H=\{0\}$. It is clear that $H\supset T_{\rho}\Sigma+\lr{H_t}$. On the other hand if $X\in (T_{\rho}\Sigma)^{\sigma}\cap\lr{H_t}^{\sigma}$ then one can write $X=\be_1H_{{\tilde b}_1}+\sum_{j=2}^k\be_jH_{b_j}$ with some $\be_j\in\RR$ hence $d{\tilde b}_1(X)=0$ which shows $H\supset (T_{\rho}\Sigma)^{\sigma}\cap\lr{H_t}^{\sigma}$ and then $T_{\rho}\Sigma+\lr{H_t}\supset H^{\sigma}$. Thus   the proof is complete.
%

\section{Propagation cone is incompatible with $\Sigma$}
\label{sec:noncompatible}

We will see in this section that transversally strictly  hyperbolic systems with the propagation cone which is incompatible with $\Sigma$  are more involved. 

\subsection{Example}
We make detailed study for the following $3\times 3$ system proposed by G.M\'etivier \cite{Me:X}
\begin{equation}
\label{eq:Met}
\begin{split}
L_a=\frac{\D}{\D t}+\left(\begin{array}{ccc}
0&1&0\\
1&0&0\\
0&0&0
\end{array}\right)\frac{\D}{\D x}+x\left(\begin{array}{ccc}
0&a&1\\
-a&0&0\\
1+a^2&0&0
\end{array}\right)\frac{\D}{\D y}\\
=\frac{\D}{\D t}+A_1\frac{\D}{\D x}+xA_2(a)\frac{\D}{\D y}=\frac{\D}{\D t}
+G_a
\end{split}
\end{equation}
where $a\in \CC$ and  $L_a$ is symmetric hyperbolic system when $a=0$.
Note that
\[
{\rm det}\,L_a(x,\tau,\xi,\eta)=\tau(\tau^2-\xi^2-x^2\eta^2)=h(x,\tau,\xi,\eta)
\]
and hence $\Sigma=\{\tau=0, \xi=0, x=0, \eta\neq 0\}$ which is  independent of $a$. Let $\rho=(\bar{t},0,\bar{y},0,0,\bar{\eta})\in \Sigma$. Since  
\[
(L_a)_{\rho}(\dot{x},{\dot\tau},{\dot\xi})={\dot\tau} I+A_1{\dot\xi}+{\bar \eta}A_2(a){\dot x}
\]
then  ${\rm det}(L_a)_{\rho}({\dot x},{\dot\tau},{\dot\xi})=h_{\rho}({\dot x},{\dot\tau},{\dot\xi})={\dot\tau}({\dot\tau}^2-{\dot\xi}^2-{\bar\eta}^2{\dot x}^2)$ is a strictly hyperbolic polynomial in $({\dot x},{\dot\tau},{\dot\xi})$. Therefore $L_a$ is transversally strictly hyperbolic system for any $a\in\CC$. It is easy to see that $C(h_{\rho})=\{t\geq \sqrt{x^2+\xi^2/{\bar \eta}^{2}}, \tau=0\}$ hence $
C(h_{\rho})\cap T_{\rho}\Sigma=\,$positive $t$ axis,  and therefore the propagation cone is incompatible with $\Sigma$.

 Replacing $x$ by $y$ or $t$ this example also yields  typical cases that the propagation cone is compatible with $\Sigma$. 
 
 Let $A_1$ and $A_2(a)$ be the same matrices given in \eqref{eq:Met}. Consider
 \[
 L_a^{(0)}=\frac{\D}{\D t}+A_1\frac{\D}{\D x}+yA_2(a)\frac{\D}{\D y}
 \]
then $ {\rm det}\,L^{(0)}_a(x,\tau,\xi,\eta)=\tau(\tau^2-\xi^2-y^2\eta^2)$ and $\Sigma$ is given by $\tau=\xi=y=0$ ($\eta\neq 0$). Therefore $L_a^{(0)}$ is transversally strictly hyperbolic system and $\Sigma$ is involutive. From  \cite[Theorem 1.6]{ MeNi:KJM} it follows that $L_a^{(0)}$ is strongly hyperbolic  for any $a\in\CC$. 

Next we consider
\[
L_a^{(1)}=\frac{\D}{\D t}+A_1\frac{\D}{\D x}+tA_2(a)\frac{\D}{\D y}
 \]
where $ {\rm det}\,L^{(1)}_a(x,\tau,\xi,\eta)=\tau(\tau^2-\xi^2-t^2\eta^2)$ and $\Sigma=\{\tau=\xi=t=0\}$ ($\eta\neq 0$) so that $L^{(1)}_a$ is transversally strictly hyperbolic system.  Since $T_{\rho}\Sigma\cap (T_{\rho}\Sigma)^{\sigma}$ is spanned by $\D/\D x$ then  $dt(T_{\rho}\Sigma\cap (T_{\rho}\Sigma)^{\sigma})=0$ is obvious. Therefore  $L_a^{(1)}$ is strongly hyperbolic near $(0,0)$ for any $a\in\CC$ by Lemma \ref{lem:dtzero} and Theorem \ref{thm:symplectic}.
%
%

\subsection{Ill-posedness}

Note
\[
L_a^*=-\frac{\D}{\D t}-A_1^*\frac{\D}{\D x}-xA_2^*(a)\frac{\D}{\D y}=-\frac{\D}{\D t}
+G^*_a
\]
and  consider the eigenvalue problem $G^*_aV(x,y)=i\be V(x,y)$. We look for 
$V(x,y)$ in the form  $V=e^{\pm iy}E^{\pm}(x)$ so that the problem is reduced to $(A_1^*\D_x\pm ix A_2^*(a))E^{\pm}(x)=-i\be E^{\pm}(x)$ where
\[
E^{\pm}(x)=\left(\begin{array}{c}
u^{\pm}(x)\\
v^{\pm}( x)\\
w^{\pm}( x)\end{array}\right).
\]
If $u^{\pm}(x)$ satisfies 
\begin{equation}
\label{eq:harmo}
(\D_x^2-x^2+\be^2\pm i {\bar a})u^{\pm}(x)=0
\end{equation}
where ${\bar a}$ stands for the complex conjugate of $a$, then with
\[
\begin{aligned}
&v^{\pm}(x)=\frac{i}{\beta}(\D_x\pm i {\bar a} x)u^{\pm}(x),\\
&w^{\pm}(x)=\mp \frac{x}{\be}u^{\pm}(x)
\end{aligned}
\]
it is easy to examine that $G^*_aV^{\pm}=i\be V^{\pm}$ where $V^{\pm}(x,y)=e^{\pm iy}E^{\pm}(x)$. Since $
G^*_aV^{\pm}(\eta x,\eta^2 y)=i\eta \be V^{\pm}(\eta x,\eta^2 y)$ one has
\[
L^*_a\big (e^{i\be \eta t}V^{\pm}(\eta x,\eta^2 y)\big)=0.
\]
The following lemma is easily checked.
\begin{lem} Assume that $ia \not\in \{z\in\RR; -1\leq z\leq 1\}$. Then either $\be^2+i{\bar a}=1$ or $\be^2-i{\bar a} =1$ has a root $\be\in \CC$ with $\im\be\neq 0$.
\end{lem}
Assume that $ia \not\in \{z\in\RR; -1\leq z\leq 1\}$ and  $\be\in\CC$ with $\im\be\neq 0$ is chosen such that $\be^2\pm i{\bar a} =1$ so that \eqref{eq:harmo} is verified by $u^{\pm}(x)=e^{-x^2/2}$ where we may assume $\im\be> 0$ without restrictions. Therefore 
\[
{\widetilde  W}_{\eta}^{\pm}(t,x,y)=\exp{\big(i\be \eta t\pm iy\eta^2-\frac{1}{2}\eta^2x^2\big)}\big(W_0+\eta xW_1^{\pm}\big)
\]
solves $L^*_a{\widetilde W}_{\eta}^{\pm}=0$ where
\[
W_0=\left(\begin{array}{c}
1\\
0\\
0\end{array}\right),\qquad W_1^{\pm}=\left(\begin{array}{c}
0\\
-i(1\mp i {\bar a})/\be\\
\mp1/\be\end{array}\right).
\]
We now consider the following Cauchy problem
\begin{equation}
\label{eq:modCauchy}
\left\{\begin{array}{lll}
L_aU=0,\\[3pt]
U(0,x,y)={ \phi}(x){ \psi}(y)W_0
\end{array}\right.
\end{equation}
where $\phi, \psi\in C_0^{\infty}(\RR)$  are real valued.  
We remark that we can assume that solutions $U$ to (\ref{eq:modCauchy}) have compact supports with respect to $(x,y)$. To examine this we recall the Holmgren uniqueness theorem (see for example \cite[Theorem 4.2]{Mi:book}). For $\delta>0$ we denote
\[
D_{\delta}=\{(t,x,y)\in \RR^{3}\mid x^2+y^2+|t|<\delta\}
\]
then we have 
\begin{prop}[Holmgren]
\label{pro:Holmgren}
There exists $\delta_0>0$ such that if $U(t,x,y)\in C^1(D_{\delta})$ with $0<\delta\leq \delta_0$ verifies 
\[
\left\{\begin{array}{ll}
L_aU=0\quad \mbox{in}\quad  D_{\delta},\\
U(0,x,y)=0\quad \mbox{on}\quad  (x,y)\in D_{\delta}\cap\{t=0\}
\end{array}\right.
\]
then $U(t,x,y)$ vanishes identically in $D_{\delta}$.
\end{prop}
\begin{theo}
\label{thm:illposed}Assume that $ia\not\in \{z\in\RR; -1\leq z\leq 1\}$. 
Let $\psi\in C_0^{\infty}(\RR)$ be an even function such that  $\psi\not\in\gamma_0^{(2)}(\RR)$ and $\phi\in C_0^{\infty}(\RR)$ with $\phi(0)\neq 0$.  Let $\Omega$ be any neighborhood of the origin of $\RR^{3}$ such that ${\rm supp}\,\phi(x)\psi(y)\subset \Omega\cap\{t=0\}$. Then  the Cauchy problem \eqref{eq:modCauchy} has no $C^1(\Omega)$ solution.
\end{theo}
\begin{proof} Suppose that there were a neighborhood $\Omega$ of the origin such that  ${\rm supp}\,\phi(x)\psi(y) \subset \Omega\cap\{t=0\}$ and the  Cauchy problem \eqref{eq:modCauchy} has a solution 
$U\in C^1(\Omega)$.
Thanks to Proposition \ref{pro:Holmgren} one can choose a small  $T>0$ such that ${\rm supp}\,U\cap  \{0\leq t\leq T\}\subset \Omega$. Denoting 
\begin{align*}
{ W}_{\eta}^{\pm}(t,x,y)=e^{-i\beta\eta T}{\widetilde W}_{\eta}(t,x,y)\\
=e^{\pm i\eta^2 y-i\beta \eta(T-t)}e^{-\eta^2x^2/2}\big(W_0+\eta xW_1^{\pm}\big)
\end{align*}
we have obviously $L^*_a{W}^{\pm}_{\eta}=0$. From
\begin{align*}
0=\int_0^T(L^*_a{W}^{\pm}_{\eta},U)dt=\int_0^T({ W}^{\pm}_{\eta},L_aU)dt\\
+( W^{\pm}_{\eta}(T),U(T))-( W^{\pm}_{\eta}(0),U(0))
\end{align*}
it follows that
\begin{equation}
\label{eq:vincenne}
(W^{\pm}_{\eta}(T),U(T))=( W^{\pm}_{\eta}(0),U(0)).
\end{equation}
Note that the left-hand side of \eqref{eq:vincenne} is $O(1)$ as $\eta\to\infty$ while the right-hand side is
\begin{equation}
\label{eq:vincenne:2}
\eta^{-1}e^{-i\beta \eta T}{\hat \psi}(\eta^2)\int e^{-x^2/2}\phi(\eta^{-1}x)dx
\end{equation}
where ${\hat \psi}$ denotes the Fourier transform of $\psi$. Then from \eqref{eq:vincenne:2} we conclude that there is $C>0$ such that for large positive $\eta$ one has
\[
|{\hat \psi}(\eta^2)|\leq C\eta e^{(-{\mathsf{Im}}\,\beta) \eta T}.
\]
Since $\psi$ is even this shows that $|{\hat \psi}(\eta)|\leq C'e^{-c\,|\eta|^{1/2}}$ with some $c>0$ and hence $\psi\in \gamma_0^{(2)}(\RR)$ which is a contradiction. 
\end{proof}
\begin{cor}
\label{cor:illposed}Assume that $ia\not\in \{z\in\RR; -1\leq z\leq 1\}$. Then the Cauchy problem for $L_a$ is $C^{\infty}$ ill-posed.
\end{cor}
%

\subsection{Well-posedness}

Consider
\[
{\tilde L}_aU=\frac{\D }{\D t}U+A_1\frac{\D}{\D x}U+\phi(x)A_2(a)\frac{\D}{\D y}U=F
\]
where $A_1$ and $A_2(a)$ are the same $3\times 3$ matrices as in \eqref{eq:Met} and $\phi(x)$ is a smooth real valued scalar function with bounded derivatives of all order and $U={^t}(u, v, w)$ and $F={^t}(f, g, h)$. In relation to $L_a$ in \eqref{eq:Met} we are interested in the case $\phi(x)=x$ in a compact neighborhood of the origin but it is not necessarily assumed here. Let $U, F\in C^1(\RR;C_0^{\infty}(\RR^d))$. Assume that  $ia\in  \{z\in\RR; -1< z< 1\}$ so that
\begin{equation}
\label{eq:sika}
a=i\mu,\quad \mu\in \RR,\quad |\mu|<1.
\end{equation}
Denote 
\[
S=\left(\begin{array}{ccc}
1&0&0\\
0&1&0\\
0&0&1/(1-\mu^2)
\end{array}\right)
\]
which is symmetric positive definite. It is easy to check that $SA_1$ and $SA_2(i\mu)$ are both hermitian, that is  $A_1$ and $A_2(i\mu)$ are {\it simultaneously} symmetrizable by $S$.  
It is clear that $(SA_1\D U/\D x,U)+(U,SA_1\D U/\D x)=0$ and
 \begin{align*}
 \big(SA_2(i\mu)\phi(x)\frac{\D U}{\D y},U\big)+\big(U, SA_2(i\mu)\phi(x)\frac{\D U}{\D y}\big)\\
 =\big(\phi(x)(SA_2(i\mu)-A_2^*(i\mu)S)\frac{\D U}{\D y},U)=0.
 \end{align*}
Then one has
\[
\frac{d}{dt}(SU,U)=\frac{d}{dt}\|S^{1/2}U\|^2=2{\mathsf{Re}}\,(SU, F)\leq 2\|S^{1/2}U\|\|S^{1/2}F\|
\]
where $(\cdot,\cdot)$ and $\|\cdot\|$ stands for the $L^2(\RR_{x,y})$ inner product and the norm respectively. 
Hence we have
\begin{equation}
\label{eq:enereq}
\frac{d}{dt}\|S^{1/2}U\|\leq \|S^{1/2}F\|.
\end{equation}
Multiplying \eqref{eq:enereq} by $e^{-\gamma t}$ and integrating such obtained  inequality 
we have
\begin{equation}
\label{eq:makino}
\begin{split}
\|S^{1/2}U(t)\|+\gamma\int_0^te^{\gamma(t-s)}\|S^{1/2}U(s)\|ds\\
\leq e^{\gamma t}\,\|S^{1/2}U(0)\|+\int_0^te^{\gamma(t-s)}\|S^{1/2}{\tilde L}_aU(s)\|ds.
\end{split}
\end{equation}
\begin{lem}
\label{lem:lessone}Assume \eqref{eq:sika}. Then for $U\in C^1(\RR; C_0^{\infty}(\RR^d))$ one has
\begin{equation}
\label{eq:enest:a}
\begin{split}
&\|U(t)\|+\gamma\int_0^te^{\gamma(t-s)}\,\|U(s)\|ds\\
&\leq \frac{1}{\sqrt{1-\mu^2}}\Big(e^{\gamma t}\,\|U(0)\|+\int_0^te^{\gamma(t-s)}\,\|{\tilde L}_aU(s)\|ds\Big).
\end{split}
\end{equation}
The same estimate holds for ${\tilde L}_a^*$.
\end{lem}
\begin{proof} The assertion for ${\tilde L}_a$ is immediate from \eqref{eq:makino} for $\|V\|\leq \|S^{1/2}V\|\leq (1-\mu^2)^{-1/2}\|V\|$. As for ${\tilde L}_a^*$ 
note that  $S^{-1}$ symmetrizes $A_1^*$ and $A_2^*(i\mu)$ simultaneously. Therefore the same arguments proving \eqref{eq:makino} shows  \eqref{eq:enest:a} for ${\tilde L}_a^*$ since $\|S^{-1/2}V\|\leq \|V\|\leq (1-\mu^2)^{-1/2}\|S^{-1/2}V\|$.
\end{proof}
 From Lemma \ref{lem:lessone} one can conclude
  \begin{theo}
  \label{thm:lessone}
 If $ia\in \{z\in\RR;-1<z<1\}$ then the Cauchy problem for ${\tilde L}_a$ is $L^2$ well-posed, and in particular,  ${\tilde L}_a$ is strongly hyperbolic. 
 \end{theo}
 \begin{proof}Denote by $H^s$ and $\|\cdot\|_s$, $s\in \RR$ the usual $L^2$ based Sobolev space of order $s$ and the norm respectively.  Assume that  $B=B(t,x,y)$ is smooth with bounded derivatives of all order and fix $T>0$ arbitrarily. Then from the standard limiting arguments, starting from \eqref{eq:enest:a}   we can conclude that for any $s\in \RR$ there is $C_s>0$ such that we have
 \begin{equation}
 \label{eq:eqL2}
 \|U(t)\|_s\leq C\Big(\|U(0)\|_s+\int_0^t\|({\tilde L}_a+B)^*U(\tau)\|_s d\tau\Big)
 \end{equation}
for any $U\in C^1([0,T];H^s)\cap C^0([0,T];H^{s+1})$.   Therefore the usual duality arguments (see for example \cite[Chapter XXIII]{Ho:book:3}) proves  that for any $U_0 \in H^s$ there exists a solution $U(t)\in C^1([0,T];H^{s-1})\cap C^0([0,T];H^{s})$ to the Cauchy problem
 \[
\left\{\begin{array}{lll}
({\tilde L}_a+B)U=0,\\[3pt]
U(0)=U_0.
\end{array}\right.
\]
The uniqueness of solution follows from \eqref{eq:enest:a}.
 \end{proof}
  \begin{theo}
  \label{thm:eqone}
If $ia\in \{z\in\RR;-1\leq z\leq 1\}$ then the Cauchy problem for ${\tilde L}_a$ is $C^{\infty}$ well-posed. 
 \end{theo}
 \begin{proof}
 Let $a=\pm i$. Then the third equation of ${\tilde L}_aU=F$ is
\[
\frac{\D w}{\D t}=h\quad \text{hence}\quad \frac{\D}{\D t}\Big(\frac{\D w}{\D y}\Big)=\frac{\D h}{\D y}
\]
from which it follows that
\begin{equation}
\label{eq:estw}
\big\|{\D w(t)}/{\D y}\big\|\leq \big\|{\D w(0)}/{\D y}\big\|+\int_0^t\big\|{\D h(s)}/{\D y}\big\|ds.
\end{equation}
Since the first two equations of ${\tilde L}_aU=F$ yield a symmetric system for $(u,v)$ where $w$ is assumed to be known, repeating similar arguments for the case $|\mu|<1$ one has
\begin{equation}
\label{eq:ane}
\frac{d}{dt}\|U\|^2=-2\,{\mathsf{Re}}\,\big(\phi \frac{\D w}{\D y},u\big)+2\,{\mathsf{Re}}\,(F,U)
\end{equation}
where the right-hand side is bounded by $
C\big(\|{\D w(t)}/{\D y}\|+\|F(t)\|\big)\|U(t)\|$. 
Therefore from  \eqref{eq:ane} one has
\begin{equation}
\label{eq:miman}
\|U(t)\|
\leq \|U(0)\|+C\int_0^t\Big(\big\|{\D w(s)}/{\D y}\big\|+\|F(s)\|\Big)ds.
\end{equation}
 Inserting  \eqref{eq:estw} into  \eqref{eq:miman} one obtains
\begin{equation}
\label{eq:imazu}
\begin{split}
\|U(t)\|\leq C'\big(\|U(0)\|+\big\|{\D U(0)}/{\D y}\big\|\big)\\
+C'\int_0^t\big(\|{\tilde L}_aU(s)\|+\big\| \frac{\D}{\D y}{\tilde L}_aU(s)\big\|\big)ds
\end{split}
\end{equation}
for $0\leq t\leq T$. We turn to ${\tilde L}^*_aU=F$. Since the first two equations of ${\tilde L}^*_aU=F$ split from the third equation and yield a symmetric system then we have with $V={^t}(u,v)$ and $G={^t}(f,g)$
\begin{equation}
\label{eq:nakasho}
\begin{split}
\|V(t)\|+\big\|{\D V(t)}/{\D y}\big\|\leq C\big(\|V(0)\|+\big\|{\D V(0)}/{\D y}\big\|\big)\\
+C\int_0^t\big(\|G(s)\|+\big\|{\D G(s)}/{\D y} \big\|\big)ds.
\end{split}
\end{equation}
From the third equation of ${\tilde L}^*_aU=F$ we have
\begin{equation}
\label{eq:asahi}
\|w(t)\|\leq \|w(0)\|+C\int_0^t\big(\|h(s)\|+\big\|{\D u(s)}/{\D y}\big\|\big)ds.
\end{equation}
Inserting \eqref{eq:nakasho} into \eqref{eq:asahi} one has
\[
\|w(t)\|\leq C'\big(\|U(0)\|+\big\|{\D U(0)}/{\D y}\big\|\big)
+C'\int_0^t\big(\|F(s)\|+\big\|{\D F(s)}/{\D y}\big\|\big)ds.
\]
Taking \eqref{eq:nakasho} into account one sees that \eqref{eq:imazu} holds also for ${\tilde L}^*_a$. Then repeating similar arguments as proving Theorem \ref{thm:lessone} one concludes that the Cauchy problem for 
 ${\tilde L}_a$ is well-posed with loss of one derivative, in particular $C^{\infty}$ well-posed.  \end{proof}
%
 %
 \begin{center}\scalebox{1}[1]{
\begin{tikzpicture}

\draw[->] (-4,0)--(4,0) node [below] {\Large${\mathsf{Re}}$};
\draw[->] (0,-3)--(0,3.5) node [right] {\Large${\mathsf{Im}}$};

\node at (1.5,2.5) {${\rm Gevrey\,} s$ ill-posed for $s>2$};
\node at (1.5,2) {($C^{\infty}$ ill-posed)};

\draw[ultra thick](0,-1)--(0,1);

\draw[->] (-1.5,-1.5)--(-0.1,0.9);
\draw[->] (-1.5,-1.5)--(-0.1,-1);

\node at (-2,-1.7) {$H^1$ well-posed};
\node at (-2,-2.2) {($C^{\infty}$ well-posed)};

\filldraw [gray] (0,1) circle (2pt);
\filldraw [gray] (0,-1) circle (2pt);

\node at (0.25,1) {\Large $i$};
\node at (0.35,-1) {\Large -$i$};

\node at (-3,3.5) {\Large $a$-plane};

\draw[<-] (0.1,0.5) arc [start angle=135, end angle=90, radius=1]
node [right] {$L^2$ well-posed};

\end{tikzpicture}
}
\end{center}
%

 \noindent
Acknowledgements: The author is very grateful to anonymous referee for pointing out some insufficient arguments in the proof of Theorem \ref{thm:symplectic} in the original version.
 

\end{document}